\documentclass{aptpub}
\authornames{T.~Vasantam and R.R.~Mazumdar}
\shorttitle{Mean field fluctuations of  loss systems under JSQ(d) load balancing}
\usepackage{amsmath}
\usepackage{amsfonts}
\usepackage{amssymb}
\usepackage{graphicx}
\usepackage{graphics}
\usepackage{graphics,color}
\usepackage{url}
\usepackage{epstopdf}
\usepackage{times}
\usepackage{comment}
\usepackage{cite}
\usepackage{ifthen}
\usepackage{tabularx}
\usepackage{tikz}
\usetikzlibrary{matrix}
\usepackage{cite}
\usepackage{tabularx}
\usepackage{bm}

\newcommand{\hide}[1]{\ifthenelse{\boolean{false}}{#1}{}}

\newcommand{\barr}{\begin{array}}
\newcommand{\earr}{\end{array}}

\newcommand{\benum}{\begin{enumerate}}
\newcommand{\eenum}{\end{enumerate}}

\newcommand{\bit}{\begin{itemize}}
\newcommand{\eit}{\end{itemize}}
\newcommand{\bmul}{\begin{multline}}
\newcommand{\emul}{\end{multline}}

\newcommand{\bdes}{\begin{description}}
\newcommand{\edes}{\end{description}}

\newcommand{\bfig}{\begin{figure}}
\newcommand{\efig}{\end{figure}}

\newcommand{\bemq}{\begin{quote} \begin{em}}
\newcommand{\eemq}{\end{em} \end{quote}}

\newcommand{\ie}{{i.e.}}

\newcommand{\expect}[1]{\mathbb{E}\left[{#1}\right]}

\newcommand{\bt}{\begin{theorem}}
\newcommand{\bl}{\begin{lemma}}
\newcommand{\bc}{\begin{claim}}
\newcommand{\bcoro}{\begin{corollary}}
\newcommand{\bres}{\begin{Result}}
\newcommand{\brem}{\begin{Remark}}
\newcommand{\et}{\end{theorem}}
\newcommand{\el}{\end{lemma}}
\newcommand{\ec}{\end{claim}}
\newcommand{\ecoro}{\end{corollary}}
\newcommand{\eres}{\end{Result}}
\newcommand{\erem}{\end{Remark}}

\newcommand{\beq}{\begin{equation}}
\newcommand{\eeq}{\end{equation}}
\newcommand{\balign}{\begin{align}}
\newcommand{\ealign}{\end{align}}
\newcommand{\nbeq}{\begin{equation*}}
\newcommand{\neeq}{\end{equation*}}
\newcommand{\nbalign}{\begin{align*}}
\newcommand{\nealign}{\end{align*}}

\newcommand{\UN}[1]{{\mathcal{V}}^{(N)}_{k}}

\newcommand{\norm}[1]{\|{#1}\|}
\newcommand{\abs}[1]{\left \vert {#1} \right \vert}
\newcommand{\mb}[1]{\mathbb{#1}}
\newcommand{\mf}[1]{\mathbf{#1}}
\newcommand{\mc}[1]{\mathcal{#1}}

\begin{document}
\title{Sensitivity of Mean-field Fluctuations in Erlang Loss Models with Randomized Routing}\\

\authorone[University of Massachusetts, Amherst]{Thirupathaiah Vasantam}
\addressone{College of Information and Computer Sciences, Amherst, MA 01003, USA. e-mail: tvasantam@cs.umass.edu}
\authortwo[University of Waterloo]{Ravi R. Mazumdar}
\addresstwo{Department of Electrical and Computer Engineering, 200 University Ave W, Waterloo, ON N2L 3G1, Canada. e-mail: mazum@uwaterloo.ca}

\begin{abstract}

In this paper, we study a large system of $N$  servers each with capacity to process at most $C$ simultaneous jobs and 
an incoming job is routed to a server if it has the lowest occupancy amongst $d$ (out of N) randomly selected servers. A job that is routed to a server with no vacancy is assumed to be blocked and lost. Such randomized policies are referred to JSQ(d) (Join the Shortest Queue out of $d$) policies.
Under the assumption that jobs arrive according to a Poisson process with rate $N\lambda^{(N)}$ where $\lambda^{(N)}=\sigma-\frac{\beta}{\sqrt{N}}$, $\sigma\in\mb{R}_+$ and $\beta\in\mb{R}$, we establish 
functional central limit theorems (FCLTs) for the fluctuation process in both the transient and stationary regimes when service time distributions are exponential. In particular, we show that the limit is an Ornstein-Uhlenbeck process whose mean and variance depend on the mean-field of the considered model. Using this, we obtain approximations to the blocking probabilities for large $N$, where we can precisely estimate the accuracy of first-order approximations. \end{abstract}

	

\section{Introduction}
\label{sec:introduction}

 This paper is motivated by the design of load balancing algorithms for cloud computing systems such as Microsoft's Azure \cite{Azure} and Amazon EC2 \cite{AmazonEC2}, where  Erlang-type loss models are the appropriate mathematical abstraction. The models of interest are large number $N$ of multi-server loss systems where each server has capacity to process at most $C$ jobs simultaneously,  where $C$ is independent of $N$. A job routed to a server will be accepted for service only if the occupancy (the number of progressing jobs) of the server is less than $C$, otherwise the routed job will be blocked from service and it is considered to be lost. If a routed job gets accepted at its destination, then its processing begins immediately at a constant unit rate until its service is completed. We assume that the service times are exponentially distributed with unit mean.

We focus on randomized routing policies where arriving jobs or tasks are routed to the server with the least number of jobs amongst $d$ servers chosen at random. These policies are referred to as JSQ(d) (Join the Shortest Queue out of $d$) policies. We consider the situation where the arrival rate to the system is given by a Poisson process with rate $N\lambda^{(N)}$ where $\lambda^{(N)}=\sigma-\frac{\beta}{\sqrt{N}}$, $\sigma\in\mb{R}_+$ and $\beta\in\mb{R}$. Thus $\lambda^{(N)}$ is a fluctuation in the arrival rate that is nominally $\sigma$. The objective in this paper is to study the sensitivity of the blocking probabilities under JSQ($d$). We also show that this result can then be exploited to obtain approximation errors.

A number of papers have studied the limiting behavior of blocking probabilities in Erlang loss models when $N\rightarrow \infty$ for both homogeneous systems \cite{xie,arpan} and heterogeneous models (servers with differing capacities) \cite{arpan2,Karthik}. When $N$ is infinite, the limiting distributions can be obtained via the mean-field limit of the empirical  occupancy distribution, an approach that goes back to the work in \cite{Vvedenskaya_inftran,mit} for $M/M/1$ models incorporating JSQ(d) and was popularized  as the {\em Power-of-Two} principle where it was argued that most of the gains in the average delay are obtained when taking $d=2$. The mean-field analysis also establishes the asymptotic statistical independence on path-space \cite{Graham_chaos} for the limiting occupancy processes. The results show that the blocking probability in the limiting system is very close to the theoretical lower bound on blocking that can be achieved by any non-anticipative policy and thus approximates the optimal desirable behavior\cite{Karthik}.

A key question is, how well does the mean-field limit describe the occupancy distribution and the blocking probabilities when $N$ is finite but large? Recently, there have been a number of papers \cite{Ying,gast} that have addressed this issue for $M/M/1$ queueing models for the $\beta=0$ case where the limiting stationary distribution can be characterized explicitly as a double-exponential distribution. They used an approach based on Stein's method and showed that the rate of convergence of the empirical occupancy distribution to the mean-field distribution is  $O(\frac{1}{\sqrt{N}})$. In \cite{gast} a refined  $O(\frac{1}{N})$ term is also given. These approaches use Stein's method and exponential stability of the underlying mean-field equation to study the mean-squared error between the empirical distribution and the mean-field limit to characterize the rate of convergence. In \cite{gast_fns}, similar proof techniques based on Stein's method have been used to show that for any twice differentiable bounded function $f(\cdot)$ the gap between $\mb{E}{[f(X^{(N)}(t))]}$ and $f(x(t))$ is $O(\frac{1}{N})$ in both transient and stationary regimes, where $X^{(N)}(t)$ and $x(t)$ are considered empirical distributions and the mean-field limit, respectively. The results of \cite{gast_fns} can be used for Erlang loss models for the special case $\beta=0$ to conclude that the error between the average blocking probability of the system with $N$ servers and the asymptotic blocking probability is $O(\frac{1}{N})$. The results of \cite{gast_fns} are however not applicable when $\beta\neq 0$ and are essentially weak convergence results.


Our approach is via the development of FCLTs for the variation of the empirical distribution around the mean-field that we term the {\em fluctuation process}. These limit theorems allow us to study both the transient and stationary fluctuations by showing convergence to an appropriate Ornstein-Uhlenbeck process whose drift and noise variance depend on the mean-field limit of the model. A by-product of our study of FCLTs is that the error between the average blocking probability of the system with $N$ servers and the asymptotic blocking probability is $O(\frac{1}{\sqrt{N}})$ for $\beta\neq 0$.

Recently, Eschenfeldt and Gamarnik\cite{gamarnik} also studied the FCLT scaling of the queue occupancy process for a system of $M/M/1$ queues in the diffusion limit with JSQ where they showed that asymptotically the distribution concentrates on queues having up to two customers. They do not consider the distributional aspects or the mean-field issues.

Clearly, $\lambda^{(N)}$ is a perturbation of $\sigma$ by $\frac{\beta}{\sqrt{N}}$. The Halfin-Whitt regime corresponds to the case when the offered load to the system is very close to the system capacity $NC$ and corresponds to the special case when we choose $\sigma=C$, \ie, jobs arrive according to a Poisson process with rate $N\lambda^{(N)}$ where $\lambda^{(N)}=C-\frac{\beta}{\sqrt{N}}$. This implies that for large $N$, $\frac{\lambda^{(N)}}{C}=1-\frac{\beta}{C\sqrt{N}}$ is very close to one implying that the system is critically loaded. We are interested in studying approximations to the blocking probability for the system when $N$ is large but a finite value when JSQ($d$) s used. 

 The interest in the Halfin-Whitt regime is because there is a phase change in the behavior of the blocking probabilities going from exponential decrease (in C) to an $\frac{1}{\sqrt{C}}$ scaling as $C$ (see \cite{gazd}) becomes large for the case of uniform routing. It can be shown that an equivalent result is valid for the case of complete resource pooling when $N$ becomes large with fixed $C$, i.e., a loss system with a single server having arrival rate of $N\lambda^{(N)}$ with server capacity $NC$. Although loss systems are stable for any finite average load, this phase change in blocking behavior has implications for dimensioning the system.

We first provide an overview of the system performance when the system is in the Halfin-Whitt regime. Clearly, the average blocking probability depends on how efficiently we use system resources. For example, let us consider the random routing case where an arrival is routed to a randomly selected server under the Halfin-Whitt regime. Then the average blocking probability experienced by an arrival is the same as in the single server loss system with capacity $C$ where the jobs arrive at a Poisson process with intensity $\lambda^{(N)}$ (due to thinning). The average blocking probability is equal to $Er(\lambda^{(N)},C)$, where $Er(\alpha,n)$ denotes the Erlang-B formula for Poisson arrivals with rate $\alpha$ and $n$ number of servers. Since $C$ is fixed, the average blocking probability converges to $Er(C,C)$ when $N\to\infty$. On the other hand, if we consider the complete resource pooling case in which an arrival is accepted for service as long as there is an empty spot at a server in the system, then the average blocking probability is given by $Er(N\lambda^{(N)},NC)$. Then from \cite{Whitt_Loss_heavytraffic},
\beq
\label{eq:pooling}
\lim_{N\to\infty}\sqrt{N}(Er(N\lambda^{(N)},NC))=\frac{\phi(\beta)}{\sqrt{C}\Phi(\beta)},
\eeq
where $\phi(\cdot)$ and $\Phi(\cdot)$ denote the density and distribution functions of a standard normal, respectively.

Now, let $P_{block}^{(N)}$ be the average blocking probability of an arrival in the system with parameter $N$ when the standard JSQ policy is used. It was shown in \cite{Mukherjee_blocking} that we obtain the same result \eqref{eq:pooling} as in the case of complete resource pooling, i.e.,
\beq
\label{eq:jsq_pool}
\lim_{N\to\infty}\sqrt{N}P_{block}^{(N)}=\frac{\phi(\beta)}{\sqrt{C}\Phi(\beta)}.
\eeq
This is to be expected since an arrival will not be blocked from service when there is an empty spot in the system similar to the case of complete resource pooling. As a result, the average blocking probability under the JSQ policy is equal to $Er(N\lambda^{(N)},NC)$. Under the influence of the JSQ($d$) policy, an arrival could be blocked from service even if there is an empty spot in the system. Therefore we expect a decrease in the system utilization when we use the JSQ($d$) policy. However, such a policy has less informational cost over the JSQ policy and is close to the optimal blocking that can be obtained with full resource sharing. It was shown in \cite{Mukherjee_blocking} that if $d$ is also scaled with $N$ denoted by $d^{(N)}$, and if $\lim_{N\to\infty}\frac{d^{(N)}}{\sqrt{N}\log(N)}=\infty$, then we still obtain \eqref{eq:jsq_pool} for the JSQ($d^{(N)}$) scheme. It is thus of interest to know what happens when $d\geq 2$ is fixed and does not scale with $N$.

Our approach is similar to the FCLT approach in \cite{Graham_clt} that was carried out for $M/M/1$ FCFS queues with the JSQ($d$) policy with $\lambda^{(N)}=b$, $b\in\mb{R}_+$. It was shown that  suitably scaled fluctuations of the stochastic empirical occupancy process around the mean-field limit converges to an Ornstein-Uhlenbeck (OU) process both in the transient and stationary regimes as $N\to\infty$. However the paper did not exploit this result further. In this paper we show how the limit theorems can be use to characterize the transient and steady-state system blocking probabilities and thus obtain approximation errors.

In preliminary work \cite{thiru_itc31}, we showed a similar FCLT result as in \cite{Graham_clt} for the case of the loss model when $\lambda^{(N)}=b$, $b\in\mb{R}_+$. We then used the FCLT limit to show that the error between $P_{block}^{(N)}$ and the asymptotic blocking probability $\pi_C^d$ is $o(N^{-\frac{1}{2}})$ where $\bm{\pi}=(\pi_i,0\leq i\leq C)$ is the fixed-point of the corresponding mean-field and $\pi_C$ is the probability that a server is fully occupied when $N\to\infty$. 

The FCLT approach has advantages over the Stein approach because it provides a process level characterization of the scaled mean squared error rather than just the rate at which the limit of the mean squared error between the approximation of the stationary distribution for fixed $N$ and the fixed point of the mean field goes to 0 as $N\to \infty$. More precisely, we show that the diffusion scaled fluctuation process for our model converges to a limit  which is an OU process with non-zero mean that depends on $\beta$ and the fixed-point $\bm{\pi}$ of the mean-field in our model. We then exploit this result to show that $\lim_{N\to\infty}\sqrt{N}(P_{block}^{(N)}-\pi_C^d)$ goes to a limit that can be explicitly characterized in terms of $\beta,\bm{\pi}$, and $C$.  We obtain results both for the transient and stationary occupancy distributions.  

It is worth pointing out that the result we obtain is interesting when the overall system is in the Halfin-Whitt regime, i.e., when $\sigma=C$. The effect of the randomized SQ($d$) routing results in individual loss servers that are also critically loaded but whose blocking cannot be obtained from the classical Halfin-Whitt blocking limit, instead the blocking is obtained from the fixed-point of the mean-field. 

The rest of the paper is organized as follows. We first introduce the system model and provide some preliminary results in Section~\ref{sec:preliminaries}.  We then give the main results of the paper in Section~\ref{sec:chfour_results} and provide their proofs in Section~\ref{sec:chfour_proofs}. Finally, we conclude the paper in Section~\ref{sec:conclusions}.

\section{System Model and Preliminaries}
\label{sec:preliminaries}

\subsection{System model}
\label{sec:system_model}
In this section, we give details of the system model. 
We study a large-scale multi-server system with $N$ Erlang loss servers and one
central job dispatcher, which routes an incoming request to one of the servers according to a predefined load balancing policy. A server accepts an incoming request if the occupancy or the number of progressing jobs of the server is less than $C$, a predefined integer value referred to as the capacity of a server. Otherwise, the request is blocked from service and it is considered to be discarded from the system. Furthermore, an accepted job is processed at a constant unit rate upon its acceptance for service
until the service of the job is completed. We assume that the service time distributions are exponential with unit mean.

The job dispatcher uses the JSQ($d$) load balancing policy defined below to dispatch the incoming jobs,
\begin{definition}{JSQ($d$) load balancing:}
	The job dispatcher routes an arriving job to the server with the least occupancy among $d$ servers selected uniformly at random. Furthermore, the ties are assumed to be broken uniformly at random. \label{def:sqd}
\end{definition}
\begin{remark}
	For our model, it does not matter whether the dispatcher samples $d$ servers with or without replacement to dispatch an arrival as both methods lead to the same asymptotic results. The proof follows by the same arguments as in \cite[pages~11-12]{Graham_clt}. Hence, we assume that the dispatcher samples $d$ servers with replacement upon an arrival to simplify the analysis.
\end{remark}

The arrival process of jobs is a Poisson process with rate $N\lambda^{(N)}$, where the parameter $\lambda^{(N)}\in\mb{R}_+$ is defined such that for $\sigma\in\mb{R}_+$ and $\beta\in\mb{R}$,
we have
\beq
\label{eq:chfour_arrival_rate}
\lambda^{(N)}=\sigma-\frac{\beta}{\sqrt{N}}.
\eeq
Clearly, $\lambda^{(N)}$ is a perturbation of $\sigma$ by $\frac{\beta}{\sqrt{N}}$.
For our model, we first show that there exists a functional law of large numbers limit referred to as the mean-field limit.  Next, we establish an FCLT result, which is exploited to quantify the error between the actual blocking probability of the system with $N$ servers and the asymptotic blocking probability expressed as a function of the unique fixed-point of the mean-field.
The particular form of $\lambda^{(N)}$ in \eqref{eq:chfour_arrival_rate} subsumes an important special case of $\sigma=C$ which corresponds to the Halfin-Whitt heavy traffic regime as the resulting traffic intensity $\rho^{(N)}=\frac{\lambda^{(N)}}{C}$ approaches one as $N\to\infty$ and $\lim_{N\to\infty}\sqrt{N}(1-\rho^{(N)})=\frac{\beta}{C}$. In this case the arrival rate of jobs $N\lambda^{(N)}$ and the system capacity $NC$ are related as $NC=N\lambda^{(N)}+\beta\sqrt{N}$, and they converge to $\infty$ as $N\to\infty$.

\subsection{Notation}
\label{sec:chfour_notation}

Since the job dispatcher uses only the knowledge of the occupancies of servers and it does not use their identities to dispatch an arrival, we consider the Markov process $(\mf{X}^{(N)}(t), t\geq 0)$ to model the time-evolution of the system where $\mf{X}^{(N)}(t)=(\mf{X}^{(N)}_i(t),0\leq i\leq C)$ with $\mf{X}^{(N)}_i(t)$ denoting the fraction of the servers with at least $i$ progressing jobs at time $t$. Let $\mb{U}$ be the space defined as
\beq
\label{eq:chfour_space}
\mb{U}\triangleq\{(u_0,u_1,\cdots,u_C):u_0=1\geq u_1\geq\cdots\geq u_C\geq 0\}.
\eeq
It is evident that $\mf{X}^{(N)}(t)$ lies in the space $\mb{U}$.
Without loss of generality, we write an element of the form $(u_0,\cdots,u_C)$ as $\mf{u}$. The space $\mb{U}$ is equipped with the metric generated by the euclidean norm $\norm{\cdot}_2$ defined as
\beq
\label{eq:chfour_euclidean_norm}
\norm{\mf{u}}_2=\sqrt{\sum_{i=0}^C\abs{u_i}^2},
\eeq
where $\mf{u}=(u_0,\cdots,u_C)$. It can be verified that the space $\mb{U}$ is a Polish space.

 We study stochastic processes that are defined on $(\Omega,\mb{F},\mb{P})$ with sample paths belonging to the space of right continuous functions with left limits, such functions are also called c\`adl\`ag functions.
The space of c\`adl\`ag functions is equipped with the Skorohod $J_1$-topology. 
 We write $\mf{Y}_n\Rightarrow \mf{Y}$ as $n\to\infty$ to indicate that a sequence of random elements $\{\mf{Y}_n\}_{n\geq 1}$ converges in distribution to a random element $\mf{Y}$.
 For two real valued local martingales $(\mf{M}^1(t),t\geq 0)$ and
 $(\mf{M}^2(t),t\geq 0)$, let the covariation process be denoted by $(<\mf{M}^1,\mf{M}^2>_t,t\geq 0)$ and the quadratic variation process be denoted by $(<\mf{M}^1>_t,t\geq 0)=(<\mf{M}^1,\mf{M}^1>_t,t\geq0)$.

\subsection{Preliminaries}
In this section, we present a mathematical modeling of the main problem and give some preliminary results.
We first begin with a discussion on the time-evolution of the process $(\mf{X}^{(N)}(t), t\geq 0)$.
 At an arrival instant $t$, if the system state is $\mf{b}=(b_0,\cdots,b_C)$ implying that the fraction of the servers with at least $i$ progressing jobs at time $t$ is equal to $b_i$ for $0\leq i\leq C$, then according to the JSQ($d$) policy the destination server of the job will have occupancy $n$ with probability $b_n^d-b_{n+1}^d$. Since the rate of the arrival process is $N\lambda^{(N)}$, the total instantaneous rate of arrivals to servers that have $n$ jobs is equal to $N\lambda^{(N)}((\mf{X}_{n-1}^{(N)}(t))^d-(\mf{X}_{n}^{(N)}(t))^d)$. Furthermore,  since the service times have exponential distributions with unit rate, the total instantaneous departure rate of jobs from servers with $n$ progressing jobs is equal to $nN(\mf{X}_{n}^{(N)}(t)-\mf{X}_{n+1}^{(N)}(t))$. As a result, we can model the time-evolution of the process $(\mf{X}^{(N)}(t), t\geq 0)$ by using random time change of a set of mutually independent unit rate Poisson processes as in \cite[Section~2.1]{pang}, which we explain below.

Let $\{(\mc{N}_i(t),t\geq 0)\}_{i\geq 1}$ be a set of mutually independent unit rate Poisson processes where $(\mc{N}_i(t),t\geq 0)$ is used to model the arrival process to servers that have $i-1$ progressing jobs. Similarly, let $\{(\mc{D}_i(t),t\geq 0)\}_{i\geq 1}$ be the collection of a set of mutually independent unit rate Poisson processes where $(\mc{D}_i(t),t\geq 0)$ is used to model the departure process from servers that have $i$ progressing jobs.
 Furthermore, the set of processes $\{(\mc{D}_i(t),t\geq 0)\}_{i\geq 1}$ is independent of the set of processes $\{(\mc{N}_i(t),t\geq 0)\}_{i\geq 1}$. Also, $\{(\mc{N}_i(t),t\geq 0)\}_{i\geq 1}$ and $\{(\mc{D}_i(t),t\geq 0)\}_{i\geq 1}$ are independent of $\mf{X}^{(N)}(0)$.
Since the arrival process of jobs to the system is a Poisson process with rate $N\lambda^{(N)}$ and the service time distributions are exponential with unit mean, we can write
\beq
\mf{X}^{(N)}_0(t)=1,
\eeq
and for $n\geq 1$,
\begin{multline}
\label{eq:chfour_evolutions}
\mf{X}_n^{(N)}(t)=\mf{X}^{(N)}_n(0)
+\frac{1}{N}\mc{N}_n\Big(N\lambda^{(N)}\int_{s=0}^t((\mf{X}_{n-1}^{(N)}(s))^d-(\mf{X}_{n}^{(N)}(s))^d)\,ds\Big)\\
-\frac{1}{N}\mc{D}_n\Big(Nn\int_{s=0}^t((\mf{X}_{n}^{(N)}(s))-(\mf{X}_{n+1}^{(N)}(s)))\,ds\Big).
\end{multline}
We choose the filtration $(\mc{F}^{(N)}(t),t\geq 0)$ where
\begin{multline}
\mc{F}^{(N)}(t)=\sigma\left(\mf{X}^{(N)}(0), \mc{N}_n\Big(N\lambda^{(N)}\int_{s=0}^r((\mf{X}_{n-1}^{(N)}(s))^d-(\mf{X}_{n}^{(N)}(s))^d)\,ds\Big),\right.\\
\left.\mc{D}_n\Big(Nn\int_{s=0}^r((\mf{X}_{n}^{(N)}(s))-(\mf{X}_{n+1}^{(N)}(s)))\,ds\Big),n\geq 1,0\leq r\leq t\right),
\end{multline}
augmented by all the null sets.

We now present the results on the mean-field analysis of the model without proofs as they directly follow from the case of $\lambda^{(N)}=b$ for $b\in\mb{R}_+$, studied in \cite{arpan}. The resulting mean-field equations (MFEs) in our case are the same as in the case of $\lambda^{(N)}=b$ except that $\sigma$ replaces $b$.

\bt
For $\mf{u}\in\mb{U}$, if $\mf{X}^{(N)}(0)\Rightarrow \mf{u}$ as $N\to\infty$, then $(\mf{X}^{(N)}(t),t\geq 0)\Rightarrow(\mf{x}(t,\mf{u}),t\geq 0)$ as $N\to\infty$ where $(\mf{x}(t,\mf{u}),t\geq 0)=(x_n(t,\mf{u}),t\geq 0,\,0\leq n\leq C)$ is the unique solution to the following equations: let $\mf{h}(\mf{x}(t,\mf{u}))=(h_n(\mf{x}(t,\mf{u})),0\leq n\leq C)$, where
\begin{align}
\mf{x}(0,\mf{u})=\mf{u},\,\,\,\,\frac{dx_n(t,\mf{u})}{dt}&=h_n(\mf{x}(t,\mf{u})),\label{eq:chfour_mfe1}
\end{align}
satisfying
\beq
\label{eq:chfour_mfe2}
h_0(\mf{x}(t,\mf{u}))=0,
\eeq
and for $n\geq 1$,
\beq
\label{eq:chfour_mfe3}
h_n(\mf{x}(t,\mf{u}))=\sigma(x_{n-1}^d(t,\mf{u})-x_n^d(t,\mf{u}))-n(x_n(t,\mf{u})-x_{n+1}(t,\mf{u}))
\eeq
with $x_0(t,\mf{u})=1$ and $x_{C+1}(t,\mf{u})=0$.
The deterministic process $(\mf{x}(t,\mf{u}),t\geq 0)$ is referred to as the mean-field limit and equations \eqref{eq:chfour_mfe1}-\eqref{eq:chfour_mfe3} are referred to as the mean-field equations with initial point $\mf{u}$.
\et
Without loss of generality, we say that a process $(\mf{y}(t),t\geq 0)$ is a solution to the differential equations $\eqref{eq:chfour_mfe1}$-$\eqref{eq:chfour_mfe3}$, it means that it is the unique generic solution with initial point $\mf{y}(0)$.

The mean-field $(\mf{x}(t,\mf{u}),t\geq 0)$ has a unique global asymptotically stable fixed-point $\boldsymbol{\pi}=(\pi_n,0\leq n\leq C)$ with $\pi_0=1$.  Also, the following exchange of limits holds
\beq
\label{eq:chfour_exchange_limits}
\lim_{N\to\infty}\lim_{t\to\infty}\mf{X}^{(N)}(t)=\lim_{t\to\infty}\lim_{N\to\infty}\mf{X}^{(N)}(t).
\eeq
Using \eqref{eq:chfour_exchange_limits}, under the assumption of the exchangeability of initial states of servers, we can show the independence of any finite set of servers as $N\to\infty$. Also, it can be shown that as $N\to\infty$ a server's distribution equals to $\mf{x}(t,\mf{u})$ and $\pi$ at time $t$ and in the stationary regime, respectively. As a result, $\pi_C$ denotes the stationary probability that a server is fully occupied as $N\to\infty$.
Since the dispatcher samples $d$ servers upon arrival and selects one of them as the destination server, the stationary average blocking probability of a job as $N\to\infty$ is equal to $\pi_C^d$, where we use the fact that the chosen $d$ servers are independent of each other. Our objective is to find the gap between the actual blocking probability of the system with $N$ servers and  $\pi_C^d$
as a function of the parameters $\beta$, $\bm{\pi}$, $\sigma$, $N$, and $C$. 

We can find $\bm{\pi}$ numerically as follows. 
The fixed-point $\bm{\pi}$ is the unique solution to the following equations
\beq
\label{eq:chfour_fixed_pt}
\sigma(\pi_{n-1}^d-\pi_n^d)=n(\pi_n-\pi_{n+1})
\eeq
for $n\geq 1$ and $\pi_{C+1}=0$.
Then from \eqref{eq:chfour_fixed_pt}, we can also write
\beq
\label{eq:chfour_fixed_pt2}
\sigma\frac{(\pi_{n-1}^d-\pi_n^d)}{(\pi_{n-1}-\pi_n)}(\pi_{n-1}-\pi_n)=n(\pi_n-\pi_{n+1})
\eeq
for $n\geq 1$ and $\pi_{C+1}=0$. Let us define $\hat{\lambda}_n=\sigma\frac{(\pi_n^d-\pi_{n+1}^d)}{(\pi_n-\pi_{n+1})}$. Then from \eqref{eq:chfour_fixed_pt2}, 
$\bm{\pi}$ is the stationary distribution of the single server loss model with a Poisson arrival process of jobs having rate $\hat{\lambda}_n$ when there are $n$ progressing jobs, and
$\pi_n$ is the probability that the server has at least $n$ progressing jobs. Let $\mb{M}_1(\{0,1,\cdots,C\})$ be the set of probability measures on $\{0,1,\cdots,C\}$. Then from \cite{arpan}, the fixed-point $\bm{\pi}$ can be computed using the formula for the stationary distribution of a single server loss system with state-dependent arrival rates.
We first define two mappings,
$\Theta:\mb{M}_1(\{0,1,\cdots,C\})\mapsto\mb{R}_+^{C+1}$ and     $\widehat{\Xi}:\mb{R}_+^{C+1}\mapsto \mb{M}_1(\{0,1,\cdots,C\})$ that are used in computing $\bm{\pi}$. For every $(p_0,\cdots,p_C)\in \mb{M}_1(\{0,1,\cdots,C\})$, there exists $(r_0,\cdots,r_C)\in\mb{R}_+^{C+1}$ such that
\beq
\Theta((p_0,\cdots,p_C))=(r_0,\cdots,r_C),
\eeq
where
\beq
r_n=\sigma\frac{((\sum_{j=n}^Cp_j)^d-(\sum_{i=n+1}^Cp_i)^d)}{((\sum_{j=n}^Cp_j)-(\sum_{i=n+1}^Cp_i))}.
\eeq
Similarly, for every  $(b_0,\cdots,b_C)\in\mb{R}_+^{C+1}$, there exists $(a_0,\cdots,a_C)\in\mb{M}_1(\{0,1,\cdots,C\})$ such that
\beq
\widehat{\Xi}((b_0,\cdots,b_C))=(a_0,\cdots,a_C),
\eeq
where
\beq
a_n=\left(\prod_{i=1}^n\Big(\frac{b_{i-1}}{i}\Big)\right)a_0
\eeq
for $n\geq 1$ and $\sum_{i=0}^Ca_i=1$. Then $\bm{\pi}$ is the unique fixed-point of the mapping $\widehat{\Xi}(\Theta)$ which can be computed numerically.


Our objective is to study the limit of the fluctuation process $(\mf{Z}^{(N)}(t),t\geq 0)$ as $N\to\infty$, where
\beq
\label{eq:chfour_z_n}
\mf{Z}^{(N)}(t)=\sqrt{N}(\mf{X}^{(N)}(t)-\mf{x}(t,\mf{u})).
\eeq
It can be checked that $\mf{Z}^{(N)}(t)$ lies in the space $\mb{V}$ defined as
\beq
\mb{V}\triangleq\{(r_0,\cdots,r_C): r_0=0 \text{ and } r_i\in\mb{R}, 1\leq i\leq C \}.
\eeq
We equip the space $\mb{V}$ with the topology induced by the euclidean norm \eqref{eq:chfour_euclidean_norm}. 
Our analysis uses the operator norm $\norm{\cdot}_2$ defined as 
\beq
\norm{K}_2=\sup_{\mf{v}\in\mb{V}}\frac{\norm{K\mf{v}}_2}{\norm{\mf{v}}_2},
\eeq
where $K:\mb{V}\mapsto\mb{V}$ is a linear operator.

Next, we obtain the time evolution of the process $(\mf{Z}^{(N)}(t),t\geq 0)$ by using \eqref{eq:chfour_evolutions}, \eqref{eq:chfour_mfe1}, and \eqref{eq:chfour_z_n}. For this, we first define the following three useful operators $W_1, W_2, W_3:\mb{U}\mapsto\mb{V}$ as follows:\\
 for $\mf{b}\in\mb{U}$ with $1\leq i\leq 3$ we define
  \beq
 W_i(\mf{b})=((W_i(\mf{b}))_n, n\geq 0), 
 \eeq
 where
\beq
(W_1(\mf{b}))_0=0,\,\, (W_2(\mf{b}))_0=0,\,\,(W_3(\mf{b}))_0=0,
\eeq
and for $1\leq n\leq C$,
\beq
\label{eq:chfour_w1w2w3}
(W_1(\mf{b}))_n=\sigma (b_{n-1}^d-b_n^d),\,\,(W_2(\mf{b}))_n=n (b_{n}-b_{n+1}),\,\,(W_3(\mf{b}))_n=\beta (b_{n-1}^d-b_n^d).
\eeq
From \eqref{eq:chfour_mfe3} and \eqref{eq:chfour_w1w2w3},  we have 
\beq
h_n(\mf{b})=(W_1(\mf{b}))_n-(W_2(\mf{b}))_n.
\eeq
Furthermore, let $W:\mb{U}\mapsto\mb{V}$ be the operator defined as
\beq
W=W_1-W_2.
\eeq
The operator $W$ is Lipschitz continuous satisfying the following inequality for all $\mf{a},\mf{b}\in\mb{U}$,
\beq
\norm{W(\mf{a})-W(\mf{b})}_2\leq B_W\norm{\mf{a}-\mf{b}}_2,
\eeq
where $B_W=2d\sqrt{\sigma^2+C^2}$.

Now we define a set of independent square-integrable martingales $(\mf{M}^{(N)}(t),t\ge 0)=\{(\mf{M}^{(N)}_i(t),t\geq 0)\}_{(i\in{0,1,\cdots,C)}}$ adapted to the filtration $(\mc{F}^{(N)}(t),t\geq 0)$ such that $(\mf{M}^{(N)}(t),t\ge 0)$ is independent of $\mf{Z}^{(N)}(0)$ and for $i\geq 1$,
\beq
\label{eq:chfour_covN}
<\mf{M}^{(N)}_i>_t=\int_{s=0}^t\left((W_1(\mf{X}^{(N)}(s)))_i+(W_2(\mf{X}^{(N)}(s)))_i-\frac{1}{\sqrt{N}}(W_3(\mf{X}^{(N)}(s)))_i\right)\,ds.
\eeq
Then from \eqref{eq:chfour_evolutions}-\eqref{eq:chfour_mfe3}, and \eqref{eq:chfour_z_n}, we get
\beq
\label{eq:chfour_SDEN2}
\mf{Z}^{(N)}(t)=\mf{Z}^{(N)}(0)+\int_{s=0}^t\sqrt{N}(W(\mf{X}^{(N)}(s))-W(\mf{x}(s,\mf{u})))\,ds
-\int_{s=0}^tW_3(\mf{X}^{(N)}(s))\,ds+\mf{M}^{(N)}(t).
\eeq

\section{Summary of Main Results}
\label{sec:chfour_results}
In this section, we give main results and provide their proofs in Section~\ref{sec:chfour_proofs}. We present the results related to the transient regime and the stationary regime in Sections~\ref{sec:chfour_transient_results} and~\ref{sec:chfour_stationary_results}, respectively.

\subsection{Transient Regime}
\label{sec:chfour_transient_results}
In this section, we show that the process $(\mf{Z}^{(N)}(t),t\geq 0)$ converges to an OU process in the transient regime as $N\to\infty$.
First, we begin with the following result that concludes stochastic boundedness of the process $(\mf{Z}^{(N)}(t),t\geq 0)$ when $N\to\infty$. We use this property in proving the tightness of the sequence $\{(\mf{Z}^{(N)}(t),t\geq 0)\}_{N\geq 1}$.
\begin{lemma}
	\label{thm:ch4_transient_bound}
	For any $T>0$, if $\limsup_{N\to\infty}\expect{\norm{\mf{Z}^{(N)}(0)}_2^2}<\infty$, then 
	\beq
	\label{eq:ch4_transient_bound}
	\limsup_{N\to\infty}\expect{\sup_{0\leq t\leq T}\norm{\mf{Z}^{(N)}(t)}_2^2}<\infty.
	\eeq
\end{lemma}
\begin{proof}
	See Section~~\ref{app:ch4_transient_bound}.
\end{proof}

As we show later in this section, any limit point of $(\mf{Z}^{(N)}(t),t\geq 0)$ is a solution to the stochastic differential equation (SDE) defined in \eqref{eq:chfour_OU}. Next, we introduce some notation that is used in \eqref{eq:chfour_OU}.
The proposed SDE depends on $(\mf{s}(t),t\geq 0)$ which is a solution of the equation \eqref{eq:chfour_linear}.
%
%
Let $(\mf{w}(t),t\geq 0)$ be a generic solution of \eqref{eq:chfour_mfe1}-\eqref{eq:chfour_mfe3} with an initial point $\mf{w}(0)$, where
$\mf{w}(t)=(w_n(t), 0\leq n\leq C+1)$ satisfying $w_0(t)=1$ and $w_{C+1}(t)=0$ . Then we have
\begin{align}
\frac{dw_n(t)}{dt}&=h_n(\mf{w}(t)),\label{eq:chfour_mfe1_sec3}
\end{align}
where
\beq
\label{eq:chfour_mfe2_sec3}
h_0(\mf{w}(t))=0 
\eeq
and for $1\leq n\leq C$,
\beq
\label{eq:chfour_mfe3_sec3}
h_n(\mf{w}(t))=\sigma(w_{n-1}^d(t)-w_n^d(t))-n(w_n(t)-w_{n+1}(t)).
\eeq

 By linearizing \eqref{eq:chfour_mfe1_sec3}-\eqref{eq:chfour_mfe3_sec3}
around a solution $(\mf{r}(t),t\geq 0)$ of \eqref{eq:chfour_mfe1_sec3}-\eqref{eq:chfour_mfe3_sec3} with an initial point $\mf{r}(0)$, we get
\beq
\label{eq:chfour_linear}
\frac{d\mf{s}(t)}{dt}=H(\mf{r}(t))\mf{s}(t),
\eeq
where for $\mf{a}\in\mb{U}$ and $\mf{b}\in\mb{V}$,  the linear operator $H(\mf{a}):\mb{V}\mapsto\mb{V}$ is defined as
\beq
(H(\mf{a})\mf{b})_n=\sigma da_{n-1}^{d-1}b_{n-1}-(\sigma d a_n^{d-1}+n)b_n+nb_{n+1},
\eeq
$n\geq 1$. Note that any solution $(\mf{s}(t), t\geq 0)$ of \eqref{eq:chfour_linear} satisfies that $\mf{s}(t)=\mf{w}(t)-\mf{r}(t)$, where $(\mf{w}(t), t\geq0)$ is a solution of the equations $\eqref{eq:chfour_mfe1_sec3}$-\eqref{eq:chfour_mfe3_sec3} with an initial point $\mf{w}(0)$.

 We will show that the limit of the sequence $\{(\mf{Z}^{(N)}(t),t\geq 0)\}_{N\geq 1}$ depends on the process $(\mf{s}(t),t\geq 0)$ when $(\mf{r}(t),t\geq 0)$ in \eqref{eq:chfour_linear} is replaced with the mean-field $(\mf{x}(t,\mf{u}),t\geq 0)$.
 The operator $H(\mf{a})$ is a matrix in the canonical basis $(0,1,0,\cdots,0)$, $(0,0,1,0,\cdots,0)$, $\ldots$, $(0,0,\cdots,0,1)$, where the dimension of each vector is $C+1$. We can write $H(\mf{a})$ as the following matrix of size $C\times C$:
\[
H(\mf{a})=
\left[ {\begin{array}{cccccc}
	-\nu_1 & 1 & 0 &0&\cdots&0\\
	\gamma_1 & 	-\nu_2&2&0&\cdots&0\\
	0& \gamma_2&-\nu_3&3&\cdots&0\\
	\vdots&\vdots&\ddots&\ddots&\ddots&\vdots\\
	0& 0&\cdots&\gamma_{C-2}&-\nu_{C-1}&C-1\\
	0&0&\cdots&0&\gamma_{C-1}&-\nu_C\\
	\end{array} } \right],
\]
where
$\gamma_i=\sigma d a_i^{d-1}$ and $\nu_i=\gamma_i+i$, $1\leq i\leq C$.

Let $(\mf{M}(t),t\geq 0)=\{(\mf{M}_i(t),t\geq 0)\}_{i\in\{0,1,\cdots,C\}}$ be a collection of mutually independent real valued continuous and centered Gaussian martingales, determined in law by their deterministic quadratic variation process
\beq
\label{eq:chfour_variation1}
<\mf{M}_n>_t=\int_{s=0}^t((W_1(\mf{x}(s,\mf{u}))_n+(W_2(\mf{x}(s,\mf{u}))_n)\,ds,
\eeq
for $n\geq 0$.
Note that both $\mf{M}(t)$ and $(<\mf{M}_i>_t,0\leq i\leq C)$ lie in $\mb{V}$.
From \eqref{eq:chfour_variation1}, the martingale $(\mf{M}(t),t\geq 0)$ is square integrable since $((W_1(\mf{b}))_n+(W_2(\mf{b}))_n)$ for $\mf{b}\in\mb{U}$ is uniformly bounded in $n$ and $\mf{b}$ due to the fact that $0\leq b_i\leq 1$, for $0\leq i\leq C$.

Now we introduce an SDE, later we show that the limit of the process $(\mf{Z}^{(N)}(t),t\geq 0)$ as $N\to\infty$ in the transient regime is a unique solution of this SDE. 
\begin{defn}{SDE for the Transient Regime:}
	Let $(\mf{Z}(t),t\geq 0)$ be a solution of the following SDE,
	\beq
	\label{eq:chfour_OU}
	\mf{Z}(t)=\mf{Z}(0)+\int_{s=0}^tH(\mf{x}(s,\mf{u}))\mf{Z}(s) \,ds-\int_{s=0}^tW_3(\mf{x}(s,\mf{u}))\,ds+\mf{M}(t).
	\eeq
\end{defn}
The solution of \eqref{eq:chfour_OU} is an OU process.
Next, we study the SDE \eqref{eq:chfour_OU} below. 
\bt
\label{thm:ch4_soln}
We show that
\benum
\item For $\mf{a}\in\mb{U}$, the linear operator $H(\mf{a})$ satisfies $\norm{H(\mf{a})}_2<B_H$, where $B_H=\sqrt{32(\sigma^2 d^2+C^2)}$.  
\item If $\expect{\norm{\mf{Z}(0)}_2^2}<\infty$, then there exists a unique strong solution to \eqref{eq:chfour_OU} denoted by $(\mf{Z}(t),t\geq 0)$ that satisfies $\expect{\sup_{t\leq T}\norm{\mf{Z}(t)}_2^2}<\infty$.
\eenum
\et
The proof of Theorem~\ref{thm:ch4_soln} follows by the similar arguments of the proof of \cite[Theorem~2]{thiru_itc31} and hence, we omit the proof.


Now we present the main result on the transient regime below. 
\begin{theorem}
	\label{thm:ch4_transient}
	If $\mf{Z}^{(N)}(0)\Rightarrow\mf{Z}(0)$, then $(\mf{Z}^{(N)}(t),t\geq 0)\Rightarrow (\mf{Z}(t),t\geq 0)$ where $(\mf{Z}(t),t\geq 0)$ is the unique solution of \eqref{eq:chfour_OU} with the initial point $\mf{Z}(0)$.
\end{theorem}
\begin{proof}
	See Section~\ref{app:ch4_transient}.
\end{proof}
\begin{remark}
	For a constant $a\in\mb{R}_+$, if $\lambda^{(N)}=a$ then 
	for $n\geq 1$,
	\nbeq
	\label{eq:chfour_w1w2w3_new}
	(W_1(\mf{b}))_n=a (b_{n-1}^d-b_n^d),\,\,(W_2(\mf{b}))_n=n (b_{n}-b_{n+1}),\,\,(W_3(\mf{b}))_n=0.
	\neeq
	As a result, from the SDE ~\eqref{eq:chfour_OU}, we recover the following SDE obtained in \cite{thiru_itc31} for the $\lambda^{(N)}=a$ case,
	\nbeq
	\label{eq:chfour_OU2}
	\mf{Z}(t)=\mf{Z}(0)+\int_{s=0}^tH(\mf{x}(s,\mf{u}))\mf{Z}(s) \,ds+\mf{M}(t).
	\neeq
\end{remark}

\subsection{Stationary Regime}
\label{sec:chfour_stationary_results}
In this section, we present results pertaining to the stationary regime. In the stationary regime, the mean-field is located at $\bm{\pi}$, and hence we assume that $\mf{u}=\bm{\pi}$.  We recall that $\bm{\pi}$ satisfies 
\beq
W(\bm{\pi})=W_1(\bm{\pi})-W_2(\bm{\pi})=0.
\eeq

Our objective is to show that the sequence of processes $\{(\mf{Z}^{(N)}(t),t\geq 0)\}_{N\geq 1}$ converges to a limit in the stationary regime as $N\to\infty$ where $(\mf{Z}^{(N)}(t),t\geq 0)$ is the process defined in \eqref{eq:chfour_z_n}. First, we will show that the sequence $\{\mf{Z}^{(N)}(t)\}_{N\geq 1}$ is relatively compact in the stationary regime by using an another 
process $(\mf{Q}^{(N)}(t),t\geq 0)$ where
\beq
\mf{Q}^{(N)}(t)=\sqrt{N}(\mf{X}^{(N)}(t)-\bm{\pi}). 
\eeq
 We introduce an SDE \eqref{eq:chfour_OU_stn} and show that there exists a unique solution to this SDE with a unique invariant law. We then use this result to prove that any limit point of the process $(\mf{Z}^{(N)}(t),t\geq 0)$ as $N\to\infty$ in the stationary regime is a stationary OU process with the same invariant law as that of the solution of the proposed SDE \eqref{eq:chfour_OU_stn}. As a result, it would imply that $\{(\mf{Z}^{(N)}(t),t\geq 0)\}_{N\geq 1}$ converges as $N\to\infty$ in the stationary regime to the solution of \eqref{eq:chfour_OU_stn}.

Next, we state the exponential stability of the mean-field in Lemma~	\ref{thm:ch4_exponential}, which we use later in the proof of the subsequent result stated in Lemma~\ref{thm:ch4_stn_bound}.
\begin{lemma}
	\label{thm:ch4_exponential}
	There exists $\delta_1>0$ and $D_3<\infty$ such that for all $\mf{u}\in\mb{U}$,
	the mean-field $(\mf{x}(t,\mf{u}),t\geq 0)$ satisfies 
	\beq
	\norm{\mf{x}(t,\mf{u})-\bm{\pi}}_2\leq e^{-\delta_1 t}D_3\norm{\mf{u}-\bm{\pi}}_2.
	\eeq
\end{lemma}
\begin{proof}
	See Section~\ref{app:ch4_exponential}.
\end{proof}

The following result shows the tightness of $\{\mf{Z}^{(N)}(t)\}_{N\geq 1}$ in the stationary regime.
\begin{lemma} 
	
	\label{thm:ch4_stn_bound}
	If $\limsup_{N\to\infty}\expect{\norm{\mf{Q}^{(N)}(0)}_2^2}<\infty$, then \beq
	\limsup_{N\to\infty}\sup_{t\geq 0}\expect{\norm{\mf{Q}^{(N)}(t)}_2^2}<\infty.
	\eeq
	Consequently, in the stationary regime corresponding to $t=\infty$, we have 
	\beq
	\limsup_{N\to\infty}\expect{\norm{\mf{Z}^{(N)}(\infty)}_2^2}<\infty.
	\eeq
\end{lemma}
\begin{proof}
	See Section~\ref{app:ch4_stn_bound}.
\end{proof}

Next, we state the SDE that is used to obtain the limit of the sequence $\{(\mf{Z}^{(N)}(t),t\geq 0)\}_{N\geq 1}$ in the stationary regime. First, we linearize \eqref{eq:chfour_mfe3} around $\bm{\pi}$ to obtain a process $(\mf{s}(t), t\geq 0)$ satisfying
\beq
\label{eq:chfour_A}
\frac{d\mf{s}(t)}{dt}=H(\bm{\pi})\mf{s}(t).
\eeq

Let $\mf{B}(t)=(\mf{B}_i(t),0\leq i\leq C)$ with $\mf{B}_0(t)=0$ where $\{(\mf{B}_i(t),t\geq 0)\}_{0\leq i\leq C}$ are independent centered Brownian motions and $\expect{\mf{B}_i^2(1)}=V_i=var(\mf{B}_i(1))=2i(\pi_i-\pi_{i+1})$, $i\geq 1$. The infinitesimal covariance matrix of $(\mf{B}(t),t\geq 0)$ is the diagonal matrix diag($\mf{V})$, where $\mf{V}=(V_n,0\leq n\leq C)$.
From \eqref{eq:chfour_variation1}, the martingales $(\mf{M}(t),t\geq 0)$ with $\mf{x}(0,\bm{\pi})=\bm{\pi}$ has the same law as $(\mf{B}(t),t\geq 0)$. 
Now we define the following SDE which is used to study the process $(\mf{Z}(t),t\geq 0)$ in the stationary regime.
\begin{defn}{An SDE for the Stationary Regime:}
	Let  $(\mf{Q}(t),t\geq 0)$ be a solution to the following SDE,
	\beq
	\label{eq:chfour_OU_stn}
	\mf{Q}(t)=\mf{Q}(0)+\int_{s=0}^tH(\bm{\pi})\mf{Q}(s)\,ds-\int_{s=0}^tW_3(\bm{\pi})\,ds+\mf{B}(t).
	\eeq
\end{defn}
Then \eqref{eq:chfour_OU_stn} defines an OU process whose drift and variance depend on $\bm{\pi}$. 

For an arbitrary $\mf{Q}(0)$ in \eqref{eq:chfour_OU_stn}, we have the following result and the proof follows by the same arguments as in the proof of Theorem~\ref{thm:ch4_soln}. Hence, we omit the proof.
\bt
\label{thm:ch4_soln_stn}
We show
\benum
\item $\norm{H(\bm{\pi})}_2$ is bounded.
\item If $\expect{\norm{\mf{Q}(0)}_2^2}<\infty$, then $(\mf{Q}(t),t\geq 0)$ given by \beq
\mf{Q}(t)=e^{H(\bm{\pi})t}\mf{Q}(0)-\int_{s=0}^te^{H(\bm{\pi})(t-s)}W_3(\bm{\pi})\,ds+\int_{s=0}^te^{H(\bm{\pi})(t-s)}\,d\mf{B}(s), \eeq
is the unique strong solution to \eqref{eq:chfour_OU_stn}.
Furthermore,
\beq
\expect{\sup_{t\leq T}\norm{\mf{Q}(t)}_2^2}<\infty.
\eeq
\eenum
\et

We point out that the transpose $H(\bm{\pi})^*$ of $H(\bm{\pi})$ is the generator of a finite state birth-death process and the birth, death, and killing rates in state $i$ ($1\leq i\leq C$) are  $\gamma_i$, $i-1$, and $1$, respectively. Let $\mc{I}$ be the identity matrix of dimension $C\times C$.
Then since $H(\bm{\pi})^*+\mc{I}$ is the generator of a birth-death process with zero killing rates, all the eigenvalues of $H(\bm{\pi})^*+\mc{I}$ are negative \cite{Ledermann}. Hence, all the eigenvalues of $H(\bm{\pi})$ are less than $-1$. As a consequence, we have the following result due to the fact that all the eigenvalues are negative.
\begin{lemma}
	\label{thm:ch4_A_stability}    
	The unique solution to \eqref{eq:chfour_A} is given by $(\mf{s}(t),t\geq 0)$ where $\mf{s}(t)=e^{H(\bm{\pi})t}\mf{s}(0)$. Furthermore,  $(\mf{s}(t),t\geq 0)$ satisfies that for some $\delta_2>0$ and $D_4<\infty$,
	\beq
	\norm{\mf{s}(t)}_2\leq e^{-\delta_2 t}D_4\norm{\mf{s}(0)}_2.
	\eeq
\end{lemma}

From Lemma~\ref{thm:ch4_A_stability} and the unique solution given in Theorem~\ref{thm:ch4_soln_stn}, the following result follows immediately.  Hence, we omit the proof.
\bt
\label{thm:ch4_invariant_stn}
The unique solution of \eqref{eq:chfour_OU_stn} as $t\to\infty$ has the invariant law coinciding with the law of a stationary Gaussian process with mean $\int_{0}^{\infty}e^{H(\bm{\pi})s}W_3(\bm{\pi})\,ds$ and covariance matrix $\int_{0}^{\infty}e^{H(\bm{\pi})s}\text{diag}(\mf{V})e^{H(\bm{\pi})^*s}\,ds$.

\et

We are now ready to state the main result on the FCLT for the stationary regime.
\bt
\label{thm:ch4_clt_stn}
Under the assumption that the system with index $N$ is in the stationary regime,  the sequence $\{(\mf{Z}^{(N)}(t),t\geq 0)\}_{N\geq 1}$ as $N\to\infty$ converges in law to the unique stationary OU process which solves \eqref{eq:chfour_OU_stn}. 

Furthermore, the limit of the sequence $\{(\mf{Z}^{(N)}(0))\}_{N\geq 1}$ in the stationary regime has the same law as the invariant law of the solution of \eqref{eq:chfour_OU_stn}.
\et
\begin{proof}
	See Section~\ref{app:ch4_clt_stn}.
\end{proof}

Now we use Theorem~\ref{thm:ch4_clt_stn} to provide an approximation to the average blocking probability in the system with $N$ servers. 
\begin{theorem}
	\label{thm:ch4_blocking}    
	Let $P_{block}^{(N)}$ be the average blocking probability in the system with $N$ servers, then
	\beq
	P_{block}^{(N)}=\pi_C^d-\frac{1}{\sigma\sqrt{N}}\left(\sum_{i=0}^C i(\kappa_i-\kappa_{i+1})\right)-\frac{\beta}{\sigma\sqrt{N}}(1-\pi_C^d)+o(N^{-\frac{1}{2}}),\label{eq:chfour_block_approx}
	\eeq
	where the vector $\mf{\kappa}=(\kappa_i,0\leq i\leq C)=\int_{0}^{\infty}e^{H(\bm{\pi})s}W_3(\bm{\pi})\,ds$ is the mean of the unique solution of \eqref{eq:chfour_OU_stn} in the stationary regime. 
\end{theorem}
\begin{proof}
	See Section~\ref{app:ch4_blocking}.
\end{proof}
\begin{remark} From Theorem~\ref{thm:ch4_blocking}, we have
	\benum
	\item If $\sigma=C$, the result \eqref{eq:chfour_block_approx} corresponds to the Halfin-Whitt regime. In this case, we have
	\nbeq
	\lim_{N\to\infty} \sqrt{N}(P_{block}^{(N)}-\pi_C^d)=-\frac{1}{\sigma}\left(\sum_{i=0}^C i(\kappa_i-\kappa_{i+1})\right)-\frac{\beta}{\sigma}(1-\pi_C^d).
	\neeq
	\item If $\beta=0$, then $W_3(\bm{\pi})=0$. As a result, 
	\nbeq
	\lim_{N\to\infty}\sqrt{N}(P_{block}^{(N)}-\pi_C^d)=0.
	\neeq
	\eenum
	The results for the case $\beta=0$ were presented in \cite{thiru_itc31}.
\end{remark}

The significance of Theorem~\ref{thm:ch4_blocking} is that although the exact blocking formula for $P_{block}^{(N)}$ is not known is also  difficult to characterize due to complex interactions between servers, when $N$ becomes large we can compute approximations to the blocking probability as a function of $\bm{\pi}$, $\beta$, $N$, $\sigma$, and $C$. 


\section{Proofs of Main Results}
\label{sec:chfour_proofs}
\subsection{Proof of Lemma~\ref{thm:ch4_transient_bound}}
\label{app:ch4_transient_bound}

We have
\beq
\label{eq:chfour_SDEN}
\mf{Z}^{(N)}(t)=\mf{Z}^{(N)}(0)+\int_{s=0}^t\sqrt{N}(W(\mf{X}^{(N)}(s))-W(\mf{x}(s,\mf{u})))\,ds-\int_{s=0}^tW_3(\mf{X}^{(N)}(s))\,ds
+\mf{M}^{(N)}(t)
\eeq
and for $n\geq 1$,
\nbeq
<\mf{M}^{(N)}_n>_t=\int_{s=0}^t((W_1(\mf{X}^{(N)}(s)))_n+(W_2(\mf{X}^{(N)}(s)))_n-\frac{1}{\sqrt{N}}(W_3(\mf{X}^{(N)}(s)))_n)\,ds.
\neeq
From \eqref{eq:chfour_SDEN}, we obtain
\nbeq
\norm{\mf{Z}^{(N)}(t)}_2\leq \norm{\mf{Z}^{(N)}(0)}_2+B_W\int_{s=0}^t\norm{\mf{Z}^{(N)}(s)}_2\,ds+2C\beta t
+\norm{\mf{M}^{(N)}(t)}_2.
\neeq
By using the Gronwall's Lemma,
\nbeq
\norm{\mf{Z}^{(N)}(t)}_2\leq (\norm{\mf{Z}^{(N)}(0)}_2
+2C\beta t+\norm{\mf{M}^{(N)}(t)}_2)e^{B_Wt}.
\neeq
As a result, we get
\nbeq
\norm{\mf{Z}^{(N)}(t)}_2^2\leq 3(\norm{\mf{Z}^{(N)}(0)}_2^2+4C^2\beta^2 t^2
+\norm{\mf{M}^{(N)}(t)}_2^2)e^{2B_Wt}.
\neeq
For $T>0$, we have
\nbeq
\sup_{0\leq t\leq T}\norm{\mf{Z}^{(N)}(t)}_2^2\leq 3(\norm{\mf{Z}^{(N)}(0)}_2^2+4C^2\beta^2T^2
+\sup_{0\leq t\leq T}\norm{\mf{M}^{(N)}(t)}_2^2)e^{2B_WT}.
\neeq
Finally, the Doob's inequality implies the following inequality
\nbeq
\expect{\sup_{0\leq t
		\leq T}\norm{\mf{Z}^{(N)}(t)}_2^2}
\leq 3e^{2B_WT}\Big(\expect{\norm{\mf{Z}^{(N)}(0)}_2^2}+4C^2\beta^2T^2
+4\expect{\sum_{i=1}^C<\mf{M}^{(N)}_i>_T}\Big).
\neeq
Since $\sup_{N\geq 1}\expect{\sum_{i=1}^C<\mf{M}^{(N)}_i>_T}<\infty$, we conclude that \nbeq\limsup_{N\to\infty}\expect{\sup_{0\leq t
		\leq T}\norm{\mf{Z}^{(N)}(t)}_2^2}<\infty. 
\neeq


\subsection{Proof of Theorem~\ref{thm:ch4_transient}}
\label{app:ch4_transient}
We recall that since the space $\mb{V}$ is a Polish space, the space of c\`adl\`ag functions under the Skhorohod topology is a Polish space\cite[Theorem~5.6,\,p.121]{Ethier_Kurtz_book}. Hence from the Prohorov's theorem \cite{Ethier_Kurtz_book}, tightness is equivalent to relative compactness. Therefore, it is enough to show the tightness and then we need to show that every limiting point has the same law as the unique OU process that solves \eqref{eq:chfour_OU} with the initial point $\mf{Z}(0)$.

We use Theorem~4.1 of \cite[page 354]{Ethier_Kurtz_book} 
to show the tightness of $(\mf{Z}^{(N)}(t),t\geq 0)$. First, we establish several useful preliminary results.

Since $\mf{Z}^{(N)}(0)\Rightarrow \mf{Z}(0)$, it implies that the sequence $\mf{Z}^{(N)}(0)$ is tight. Let $\tilde{B}(r)$ be the closed ball with radius $r$ centered at $\mf{0}$. For every $\epsilon>0$, there exists $r_{\epsilon}<\infty$ such that $\mb{P}(\mf{Z}^{(N)}(0)\in \tilde{B}(r_{\epsilon}))>1-\epsilon$ for all $N\geq 1$. We now define a random variable $\mf{X}^{(N,\epsilon)}(0)$ such that it coincides with $\mf{X}^{(N)}(0)$ on $\{\mf{Z}^{(N)}(0)\in \tilde{B}(r_{\epsilon})\}$ and $\mf{Z}^{(N,\epsilon)}(0)$ is uniformly bounded in $N$ on $\{\mf{Z}^{(N)}(0)\notin \tilde{B}(r_{\epsilon})\}$. Then by using coupling arguments, the processes $(\mf{Z}^{(N,\epsilon)}(t),t\geq 0)$ and $(\mf{Z}^{(N)}(t),t\geq 0)$ coincide on $\{\mf{Z}^{(N)}(0)\in \tilde{B}(r_{\epsilon})\}$. Hence, without loss of generality, we assume that $\mf{Z}^{(N)}(0)$ is uniformly bounded in $N$. As a consequence, the result stated in Lemma~\ref{thm:ch4_transient_bound} can be used in the rest of the proof.

Next, we state the following useful result from \cite[Lemma~3.3]{Graham_clt}. For $a$ and $h$ in $\mb{R}$, let 
\nbeq
\hat{B}(a,h)=(a+h)^d-a^d-da^{d-1}h.
\neeq 
Then if both $a$ and $a+h$ lie in $[0,1]$, we have
\beq
\label{eq:chfour_B}
0\leq \hat{B}(a,h)\leq h^d+(2^d-d-2)ah^2.
\eeq 
We now define a mapping $\hat{G}:\mb{U}\times\mb{V}\mapsto\mb{V}$ as follows: for $\mf{r}\in\mb{U}$ and $\mf{y}\in\mb{V}$,
\beq
\label{eq:chfour_G}
(\hat{G}(\mf{r},\mf{y}))_n=\sigma \hat{B}(r_{n-1},y_{n-1})-\sigma \hat{B}(r_{n},y_{n}).
\eeq
Then if $\mf{r}+\mf{y}\in\mb{U}$, we have
\beq
\label{eq:chfour_WHG}
W(\mf{r}+\mf{y})-W(\mf{r})=H(\mf{r})\mf{y}+\hat{G}(\mf{r},\mf{y}).
\eeq

Note that since $\mf{Z}^{(N)}(t)=\sqrt{N}(\mf{X}^{(N)}(t)-\mf{x}(t,\mf{u}))$, we have
\beq
\mf{X}^{(N)}(t)=\mf{x}(t,\mf{u})+\frac{\mf{Z}^{(N)}(t)}{\sqrt{N}}.
\eeq
Here, $\mf{X}^{(N)}(t)$, $\mf{x}(t,\mf{u})\in\mb{U}$ and $\frac{\mf{Z}^{(N)}(t)}{\sqrt{N}}\in\mb{V}$. Hence, we can write the following equation from \eqref{eq:chfour_WHG},
\beq
\label{eq:chfour_WHGN}
W(\mf{X}^{(N)}(t)-W(\mf{x}(t,\mf{u}))
=H(\mf{x}(t,\mf{u}))\frac{\mf{Z}^{(N)}(t)}{\sqrt{N}}
+\hat{G}\left(\mf{x}(t,\mf{u}),\frac{\mf{Z}^{(N)}(t)}{\sqrt{N}}\right).
\eeq

Now we show that the conditions of \cite[Theorem~4.1, page~354]{Ethier_Kurtz_book} are satisfied.
Let us write $\mf{B}_n(t)$ and $\mf{A}_n^{i,j}(t)$ of \cite[Theorem~4.1, page~354]{Ethier_Kurtz_book} as $\mf{D}_{(N)}(t)$ and $\mf{A}_{(N)}^{i,j}(t)$, respectively. Then from \eqref{eq:chfour_SDEN2} we write
\beq
\mf{D}_{(N)}(t)=\int_{s=0}^t\sqrt{N}(W(\mf{X}^{(N)}(s))-W(\mf{x}(s,\mf{u})))\,ds
-\int_{s=0}^tW_3(\mf{X}^{(N)}(s))\,ds.
\eeq
From \eqref{eq:chfour_WHGN}, we can also write
\begin{multline}
\label{eq:chfour_DN}
\mf{D}_{(N)}(t)\\
\quad=\int_{s=0}^tH(\mf{x}(s,\mf{u}))\mf{Z}^{(N)}(s)\,ds+\int_{s=0}^t\sqrt{N}\hat{G}\left(\mf{x}(s,\mf{u}),\frac{\mf{Z}^{(N)}(s)}{\sqrt{N}}\right)\,ds
-\int_{s=0}^tW_3(\mf{X}^{(N)}(s))\,ds.
\end{multline}

Also, from \eqref{eq:chfour_covN} we have
\beq
\mf{A}_{(N)}^{i,i}(t)=\int_{s=0}^t\left((W_1(\mf{X}^{(N)}(s)))_i+(W_2(\mf{X}^{(N)}(s)))_i-\frac{1}{\sqrt{N}}(W_3(\mf{X}^{(N)}(s)))_i\right)\,ds,
\eeq
and $\mf{A}_{(N)}^{i,j}(t)=0$ for $i\neq j$. Let us define
\beq
\mf{D}(t)=\int_{s=0}^tH(\mf{x}(s,\mf{u}))\mf{Z}(s)\,ds
-\int_{s=0}^tW_3(\mf{x}(s,\mf{u}))\,ds
\eeq
and 
\beq
A^{i,i}(t)=\int_{s=0}^t\left((W_1(\mf{x}(s,\mf{u})))_i+(W_2(\mf{x}(s,\mf{u})))_i\right)\,ds,
\eeq
with $A^{i,j}(t)=0$ for $i\neq j$.

Since $(\mf{Z}_i^{(N)}(t),t\geq 0)$ has jumps of size $\frac{1}{\sqrt{N}}$ and from the continuity of  $\mf{D}_{(N)}(t)$ and $\mf{A}_{(N)}^{i,j}(t)$ in $t$, the conditions $(4.3)$-$(4.5)$ of \cite[Theorem~4.1, p.~354]{Ethier_Kurtz_book} are valid. From the condition $(4.6)$ of \cite[Theorem~4.1, p.~354]{Ethier_Kurtz_book}, we need to show that for $T>0$, $\mf{V}_{(N)}(T)\to 0$ in probability as $N\to\infty$, where $\mf{V}_{(N)}(t)=\mf{D}_{(N)}(t)-\mf{D}(t)$ is given by
\begin{multline}
\mf{V}_{(N)}(t)
\\\quad=\left(\int_{s=0}^tH(\mf{x}(t,\mf{u}))\mf{Z}^{(N)}(t)\,ds+\int_{s=0}^t\sqrt{N}\hat{G}\left(\mf{x}(t,\mf{u}),\frac{\mf{Z}^{(N)}(t)}{\sqrt{N}}\right)\,ds
-\int_{s=0}^tW_3(\mf{X}^{(N)}(s))\,ds\right)\\
-\left(\int_{s=0}^tH(\mf{x}(s,\mf{u}))\mf{Z}^{(N)}(s)\,ds
-\int_{s=0}^tW_3(\mf{x}(s,\mf{u}))\,ds\right).
\end{multline}

From \eqref{eq:ch4_transient_bound} and \eqref{eq:chfour_B}, we have $\int_{s=0}^T\sqrt{N}\hat{G}\left(\mf{x}(s,\mf{u}),\frac{\mf{Z}^{(N)}(s)}{\sqrt{N}}\right)\,ds\to 0$ in probability. From the existence of the mean-field limit, it follows that $\int_{s=0}^TW_3(\mf{X}^{(N)}(s))\,ds-\int_{s=0}^TW_3(\mf{x}(s,\mf{u}))\,ds\to 0$ in probability.
This concludes that the condition $(4.6)$ of \cite[Theorem~4.1, p.~354]{Ethier_Kurtz_book} is also true. The condition $(4.7)$ of \cite[Theorem~4.1, p.~354]{Ethier_Kurtz_book} is also true since $\mf{A}_{(N)}^{i,i}(T)-A^{i,i}(T)\to 0$ in probability from the existence of the mean-field limit.
Finally, from the fact that the SDE \eqref{eq:chfour_OU} has a unique solution, the sequence $
\{(\mf{Z}^{(N)}(t),t\geq 0)\}_{N\geq 1}$ converges to the unique solution of the SDE \eqref{eq:chfour_OU} as $N\to\infty$. 
	
	\subsection{Proof of Lemma~\ref{thm:ch4_stn_bound}}
	\label{app:ch4_stn_bound}
	The proof is based on Lemma~\ref{thm:ch4_exponential}. Although our main objective is to establish the convergence of $\{\mf{Z}^{(N)}(t)\}_{N\geq 1}$ as $N\to\infty$ in the stationary regime, we prove the tightness of the stationary sequence $\{\mf{Z}^{(N)}(\infty)\}_{N\geq 1}$ by studying an alternative process $(\mf{Q}^{(N)}(t),t\geq 0)$ with the help of Lemma~\ref{thm:ch4_exponential}, where
	\nbeq
	\label{eq:chfour_qn_defn}
	\mf{Q}^{(N)}(t)=\sqrt{N}(\mf{X}^{(N)}(t)-\bm{\pi}).
	\neeq
	Let us write the solution to the mean-field equation \eqref{eq:chfour_mfe3} at time $h$ with the initial point $\mf{v}$ as $\mf{y}(h,\mf{v})$. We have
	\beq
	\label{eq:chfour_mf_random}
	\mf{y}(h,\mf{v})=\mf{v}+\int_{s=0}^hW(\mf{y}(s,\mf{v}))\,ds.
	\eeq
	Also, the process $(\mf{X}^{(N)}(t),t\geq 0)$ satisfies
	\beq
	\label{eq:chfour_evolution}
	\mf{X}^{(N)}(t)=\mf{X}^{(N)}(0)+\int_{s=0}^tW(\mf{X}^{(N)}(s))\,ds-\int_{s=0}^t\frac{1}{\sqrt{N}}W_3(\mf{X}^{(N)}(s))\,ds+\frac{\mf{M}^{(N)}(t)}{\sqrt{N}}.
	\eeq

	For $t_0\geq 0$, we obtain
	\nbeq
	\mf{Q}^{(N)}(t_0+h)=\sqrt{N}(\mf{X}^{(N)}(t_0+h)-\mf{y}(h,\mf{X}^{(N)}(t_0)))\
	+\sqrt{N}(\mf{y}(h,\mf{X}^{(N)}(t_0))-\bm{\pi}).
	\neeq
	By defining $\mf{Z}^{(N)}(t_0, h)=\sqrt{N}(\mf{X}^{(N)}(t_0+h)-\mf{y}(h,\mf{X}^{(N)}(t_0)))$, we obtain
	\beq
	\label{eq:chfour_qnzn}
	\mf{Q}^{(N)}(t_0+h)=\mf{Z}^{(N)}(t_0, h)
	+\sqrt{N}(\mf{y}(h,\mf{X}^{(N)}(t_0))-\bm{\pi}).
	\eeq
	Also, from \eqref{eq:chfour_mf_random},
	\beq
	\label{eq:chfour_zn}
	\mf{Z}^{(N)}(t_0,h)=\sqrt{N}(\mf{X}^{(N)}(t_0+h)-\mf{X}^{(N)}(t_0))-\sqrt{N}\int_{s=0}^t W(\mf{y}(s,\mf{X}^{(N)}(t_0)))\,ds.
	\eeq
	From \eqref{eq:chfour_evolution}, we have
	\begin{multline}
	\label{eq:chfour_xn_difference}
	\mf{X}^{(N)}(t_0+h)-\mf{X}^{(N)}(t_0)=\int_{s=t_0}^{t_0+h}W(\mf{X}^{(N)}(s))\,ds+\frac{\mf{M}^{(N)}(t_0+h)-\mf{M}^{(N)}(t_0)}{\sqrt{N}}\\
	-\frac{1}{\sqrt{N}}\int_{s=t_0}^{t_0+h}W_3(\mf{X}^{(N)}(s))\,ds.
	\end{multline}
	Then by using \eqref{eq:chfour_zn} and \eqref{eq:chfour_xn_difference}, we get
	\begin{multline}
	\label{eq:chfour_zn_total}
	\mf{Z}^{(N)}(t_0,h)=\sqrt{N}\int_{s=0}^{h}W(\mf{X}^{(N)}(t_0+s))\,ds-\sqrt{N}\int_{s=0}^{h}W(\mf{y}(s,\mf{x}^{(N)}(t_0)))\,ds\\ +(\mf{M}^{(N)}(t_0+h)-\mf{M}^{(N)}(t_0))
	-\int_{s=t_0}^{t_0+h}W_3(\mf{X}^{(N)}(s))\,ds.\nonumber
	\end{multline}
	After simplifications, we can write
	\begin{multline}
	\norm{\mf{Z}^{(N)}(t_0,h)}_2\leq\sqrt{N}\int_{s=0}^h\norm{W(\mf{X}^{(N)}(t_0+s))-W(\mf{y}(s,\mf{X}^{(N)}(t_0)))}_2\,ds\\
	+\norm{\mf{M}^{(N)}(t_0+h)-\mf{M}^{(N)}(t_0))}_2+\abs{\beta} d\sqrt{C}h.\nonumber
	\end{multline}
	Hence, we obtain
	\nbeq
	\norm{\mf{Z}^{(N)}(t_0,h)}_2\leq B_W\int_{s=0}^h\norm{\mf{Z}^{(N)}(t_0,s)}_2\,ds+\abs{\beta}d\sqrt{C}h+\norm{\mf{M}^{(N)}(t_0+h)-\mf{M}^{(N)}(t_0))}_2.
	\neeq
	For any $T\geq 0$, the Gronwall's inequality implies that there exists a constant $S_T$ such that
	\beq
	\label{eq:chfour_zn_final}
	\sup_{0\leq h\leq T}\norm{\mf{Z}^{(N)}(t_0,h)}_2\leq S_T(S_T+\sup_{0\leq h\leq T}\norm{\mf{M}^{(N)}(t_0+h)-\mf{M}^{(N)}(t_0))}_2).
	\eeq
	By using Lemma~\ref{thm:ch4_exponential},  we can write from \eqref{eq:chfour_qnzn} as below 
	\beq
	\label{eq:chfour_rel1}
	\norm{\mf{Q}^{(N)}(t_0+h)}_2\leq\norm{\mf{Z}^{(N)}(t_0,h)}_2+e^{-\delta_1 h}D_3\norm{\mf{Q}^{(N)}(t_0)}_2.
	\eeq
	From \eqref{eq:chfour_zn_final} and \eqref{eq:chfour_rel1}, we obtain
	\nbeq
	\norm{\mf{Q}^{(N)}(t_0+h)}_2\leq S_T(S_T+\sup_{0\leq h\leq T}\norm{\mf{M}^{(N)}(t_0+h)-\mf{M}^{(N)}(t_0))}_2)+e^{-\delta_1 h}D_3\norm{\mf{Q}^{(N)}(t_0)}_2.
	\neeq
	
	As a result, we can find some constant $L_T$ as a function of $T$ such that for $0\leq h\leq T$, we have
	\beq
	\label{eq:chfour_rel4}
	\expect{\norm{\mf{Q}^{(N)}(t_0+h)}_2^2}
	\leq L_T
	+3e^{-2\delta_1h}D_3^2\expect{\norm{\mf{Q}^{(N)}(t_0)}^2_2}.
	\eeq

	We now select a large value of $T$ such that $3e^{-2\delta_1 T}D_3^2\leq \epsilon<1$. Then for all $N\geq 1$ and an integer $m$, we get
	\nbeq
	\expect{\norm{\mf{Q}^{(N)}((m+1)T)}_2^2}\leq L_T+\epsilon \expect{\norm{\mf{Q}^{(N)}(mT)}^2_2}.
	\neeq
	By using the induction method, we obtain
	\begin{align}
	\label{eq:chfour_rel5}
	\expect{\norm{\mf{Q}^{(N)}(mT)}_2^2}&\leq L_T(\sum_{j=1}^m\epsilon^{j-1})+\epsilon^m \expect{\norm{\mf{Q}^{(N)}(0)}^2_2}\nonumber\\
	&\leq\frac{L_T}{1-\epsilon}+\expect{\norm{\mf{Q}^{(N)}(0)}^2_2}.
	\end{align}
	However, from \eqref{eq:chfour_rel4},
	\nbeq
	\sup_{0\leq h\leq T}\expect{\norm{\mf{Q}^{(N)}(mT+h)}_2^2}
	\leq L_T+3D_3^2 \expect{\norm{\mf{Q}^{(N)}(mT)}^2_2}.
	\neeq
	As a consequence, \eqref{eq:chfour_rel5} implies
	\nbeq
	\sup_{0\leq h\leq T}\expect{\norm{\mf{Q}^{(N)}(mT+h)}_2^2}
	\leq L_T
	+3D_3^2 \left(\frac{L_T}{1-\epsilon}+\expect{\norm{\mf{Q}^{(N)}(0)}^2_2}\right).
	\neeq
	Since $m$ is arbitrary, we conclude
	\nbeq
	\sup_{t\geq 0}\expect{\norm{\mf{Q}^{(N)}(t)}_2^2}
	\leq L_T
	+3D_3^2 \left(\frac{L_T}{1-\epsilon}+\expect{\norm{\mf{Q}^{(N)}(0)}^2_2}\right).
	\neeq
	
	From ergodicity and the Fatou's Lemma~\cite[p.~492]{Ethier_Kurtz_book}, the stationary random variable $\mf{Z}^{(N)}(\infty)$ satisfies
	\begin{align}
	\expect{\norm{\mf{Z}^{(N)}(\infty)}_2^2}&\leq \liminf_{t\geq 0}\expect{\norm{\mf{Q}^{(N)}(t)}_2^2}\nonumber\\
	&\leq \sup_{t\geq 0}\expect{\norm{\mf{Q}^{(N)}(t)}_2^2}.\nonumber
	\end{align}
	Finally, to show that $\limsup_{N\to\infty}\expect{\norm{\mf{Z}^{(N)}(\infty)}_2^2}<\infty$, we need to find an $\mf{X}^{(N)}(0)$ such that $\limsup_{N\to\infty}\expect{\norm{\mf{Q}^{(N)}(0)}^2_2}<\infty$. For $n\geq 1$, if we select $\mf{X}^{(N)}_n(0)=\frac{j}{N}$ so that $\frac{-1}{2N}\leq \pi_n-\frac{j}{N}\leq \frac{1}{2N}$, then $\expect{\norm{\mf{Q}^{(N)}(0)}^2_2}<\frac{C}{4N}$.
	Hence, $\limsup_{N\to\infty}\expect{\norm{\mf{Q}^{(N)}(0)}^2_2}=0$. This completes the proof.
	
	\subsection{Proof of Theorem~\ref{thm:ch4_clt_stn}}
	\label{app:ch4_clt_stn}
	From Lemma~\ref{thm:ch4_stn_bound} and the Markov inequality, the sequence $\{\mf{Z}^{(N)}(0)\}_{N\geq 1}$ is tight. As a result, from the Prohorov theorem~\cite[Page~104]{Ethier_Kurtz_book}, the sequence $\{\mf{Z}^{(N)}(0)\}_{N\geq 1}$ is relatively compact. Consider a converging subsequence and let $\mf{Z}^{(\infty)}(0)$ be its limiting point, which is square integrable. Then from Theorem~\ref{thm:ch4_transient}, the considered converging subsequence converges in law to the unique OU process $(\mf{Z}^{(\infty)}(t),t\geq 0)$ that solves the SDE \eqref{eq:chfour_OU_stn} with initial point $\mf{Z}^{(\infty)}(0)$. But, we know from \cite[Lemma~7.7 and Theorem~7.8, page 131]{Ethier_Kurtz_book} that the limit of a sequence of stationary processes is stationary. Hence the law of $(\mf{Z}^{(\infty)}(t),t\geq 0)$ should be the unique law of the stationary OU process solving the SDE \eqref{eq:chfour_OU_stn}. This argument applies for every converging subsequence. Hence, the sequence $\{(\mf{Z}^{(N)}(t),t\geq 0)\}_{N\geq 1}$ in the stationary regime converges to the unique stationary OU process solving the SDE \eqref{eq:chfour_OU_stn}. This completes the proof.
	\subsection{Proof of Theorem~\ref{thm:ch4_blocking}}
	\label{app:ch4_blocking}
	The proof is based on Little's law~\cite[Theorem~4.1]{asmussen}, Theorem~\ref{thm:ch4_invariant_stn}, and Theorem~\ref{thm:ch4_clt_stn}. Let us consider a random variable $\mf{\tilde{S}}^{(N)}(\infty)$ which denotes the number of progressing jobs in the system in stationary.
	From Little's law, we obtain
	\nbeq
	(N\lambda^{(N)})(1-P_{block}^{(N)})=\expect{\mf{\tilde{S}}^{(N)}(\infty)}.
	\neeq
	We now consider a random variable $\mf{X}^{(N)}(\infty)=(\mf{X}_n^{(N)}(\infty),0\leq n\leq C)$ where $\mf{X}_{n}^{(N)}(\infty)$ is the fraction of servers with at least $n$ progressing jobs in the stationary regime. 
	Then we can write
	\nbeq
	\mf{\tilde{S}}^{(N)}(\infty)=\sum_{n=0}^C Nn( \mf{X}^{(N)}_n(\infty)-\mf{X}^{(N)}_{n+1}(\infty)).
	\neeq
	Hence, we have
	\nbeq
	(N\lambda^{(N)})(1-P_{block}^{(N)})=\expect{\sum_{n=0}^C Nn(\mf{X}^{(N)}_n(\infty)-\mf{X}^{(N)}_{n+1}(\infty))}.
	\neeq
	Therefore, we obtain
	\beq
	\label{eq:chfour_gap1}
	\lambda^{(N)}(1-P_{block}^{(N)})=\sum_{n=0}^C n \expect{\mf{X}^{(N)}_n(\infty)-\mf{X}^{(N)}_{n+1}(\infty)}.
	\eeq
	Also, from Theorem~\ref{thm:ch4_invariant_stn}, Lemma~\ref{thm:ch4_stn_bound}, and Theorem~\ref{thm:ch4_clt_stn}, since the diffusion limit in the stationary regime has the mean vector $\kappa=\int_{0}^{\infty}e^{H(\bm{\pi})s}W_3(\bm{\pi})\,ds$ and $\limsup_{N\to\infty}\expect{\norm{\mf{Z}^{(N)}(\infty)}^2}<\infty$, we have
	\nbeq
	\lim_{N\to\infty}\sqrt{N}\left(\expect{\mf{X}^{(N)}(\infty)}-\bm{\pi}\right)=\kappa.
	\neeq
	Therefore
	\beq
	\label{eq:chfour_gap2}
	\expect{\mf{X}^{(N)}(\infty)}-\bm{\pi}=\frac{\kappa}{\sqrt{N}}+
	o(N^{-\frac{1}{2}}).
	\eeq
	
	Then from \eqref{eq:chfour_gap1} and \eqref{eq:chfour_gap2}, we get
	\beq
	\label{eq:chfour_total_eq}
	\lambda^{(N)}(1-P_{block}^{(N)})=\sum_{n=0}^C n(\pi_n-\pi_{n+1})+\frac{1}{\sqrt{N}}\sum_{n=0}^C n(\kappa_n-\kappa_{n+1})+o(N^{-\frac{1}{2}}).
	\eeq
	But, from the stationary mean-field equations, the fixed-point satisfies
	\beq
	\label{eq:chfour_fp}
	\pi_0=1 \text{ and }\sum_{n=0}^C n(\pi_n-\pi_{n+1})=\sigma(1-\pi_C^d).
	\eeq
	Then from \eqref{eq:chfour_total_eq} and \eqref{eq:chfour_fp}, we can write
	\nbeq
	1-P_{block}^{(N)}=\frac{1}{\lambda^{(N)}}\left[\sigma(1-\pi_C^d)+\frac{1}{\sqrt{N}}\sum_{n=0}^C n(\kappa_n-\kappa_{n+1})\right]+o(N^{-\frac{1}{2}}).
	\neeq
	However, by using $\lambda^{(N)}=\sigma-\frac{\beta}{\sqrt{N}}$ and the fact that $\frac{\beta}{\sigma\sqrt{N}}<1$, we obtain
	\beq
	1-P_{block}^{(N)}=\frac{1}{\sigma}\left(1+\frac{\beta}{\sigma\sqrt{N}}+o(N^{-\frac{1}{2}})\right)\left[\sigma(1-\pi_C^d)+\frac{1}{\sqrt{N}}\sum_{n=0}^C n(\kappa_n-\kappa_{n+1})\right]
	+o(N^{-\frac{1}{2}}).\nonumber
	\eeq
	After simple calculations, we obtain
	\nbeq
	P_{block}^{(N)}=\pi_C^d-\frac{1}{\sigma\sqrt{N}}\left(\sum_{n=0}^C n(\kappa_n-\kappa_{n+1})\right)-\frac{\beta}{\sigma\sqrt{N}}(1-\pi_C^d)+o(N^{-\frac{1}{2}}).
	\neeq

	This completes the proof.
	\subsection{Proof of Lemma~\ref{thm:ch4_exponential}}
	\label{app:ch4_exponential}
	The proof uses the quasi-monotonicity of the mean-field. Let us write the unique solution to the MFEs with the initial point $\mf{v}$ as $(\mf{y}(t,\mf{v}),t\geq 0)$. From the quasi-monotonicity of the mean-field, we have
	\beq
	\label{eq:chfour_monotone_eqns}
	\mf{y}(t,\min(\mf{v},\bm{\pi}))\leq \mf{y}(t,\mf{v})\leq \mf{y}(t,\max(\mf{v},\bm{\pi})).
	\eeq
	For $\mf{a},\mf{b}\in\mb{U}$, let
	\nbeq
	\norm{\mf{a}-\mf{b}}_1=\sum_{i=0}^C\abs{a_i-b_i}.
	\neeq
	From Lemma~4 of \cite{xie}, since  $\min(\mf{v},\bm{\pi})\leq \bm{\pi}$ and $\max(\mf{v},\bm{\pi})\geq \bm{\pi}$, we have
	\begin{eqnarray}
	\displaystyle \norm{\mf{y}(t,\min(\mf{v},\bm{\pi}))-\bm{\pi}}_1\leq e^{-t}\norm{\min(\mf{v},\bm{\pi})-\bm{\pi}}_1,\nonumber\\
	\displaystyle \norm{\mf{y}(t,\max(\mf{v},\bm{\pi}))-\bm{\pi}}_1\leq e^{-t}\norm{\max(\mf{v},\bm{\pi})-\bm{\pi}}_1.\label{eq:chfour_gas_bounds}
	\end{eqnarray}
	For $t\geq 0$, let us define two sets $V_+(t)$ and $V_-(t)$ as 
	\begin{eqnarray}
	V_+(t)=\{i:y_i(t,\mf{v})\geq \pi_i\},\nonumber\\
	V_-(t)=\{i:y_i(t,\mf{v})< \pi_i\}.\nonumber
	\end{eqnarray}
	Then we can write
	\nbeq
	\norm{\mf{y}(t,\mf{v})-\bm{\pi}}_1=\sum_{i\in V_+(t)}(y_i(t,\mf{v})-\pi_i)+\sum_{j\in V_-(t)}(\pi_j-y_j(t,\mf{v})).
	\neeq
	From \eqref{eq:chfour_monotone_eqns} and \eqref{eq:chfour_gas_bounds}, we write
	\begin{align}
	\norm{\mf{y}(t,\mf{v})-\bm{\pi}}_1&\leq \sum_{i\in V_+(t)}(y_i(t,\max(\mf{v},\bm{\pi}))-\pi_i)+\sum_{j\in V_-(t)}(\pi_j-y_j(t,\min(\mf{v},\bm{\pi}))),\nonumber\\
	&\leq\norm{\mf{y}(t,\max(\mf{v},\bm{\pi}))-\bm{\pi}}_1+\norm{\mf{y}(t,\min(\mf{v},\bm{\pi}))-\bm{\pi}}_1,\nonumber\\
	&\leq e^{-t}\norm{\max(\mf{v},\bm{\pi})-\bm{\pi}}_1+e^{-t}\norm{\min(\mf{v},\bm{\pi})-\bm{\pi}}_1,\nonumber\\
	&= e^{-t}\norm{\mf{v}-\bm{\pi}}_1.\nonumber
	\end{align}
	Finally, the result follows from the fact that the norms $\norm{\cdot}_1$ and $\norm{\cdot}_2$ are equivalent.
	
	
	\section{Conclusions}
	\label{sec:conclusions}
	In this paper, we established an FCLT satisfied by the fluctuation process around the mean-field of the occupancy distribution both in transient and stationary regimes for loss models with exponential service time distributions under the assumption that $\lambda^{(N)}=\sigma-\frac{\beta}{\sqrt{N}}$. The proof used the global stability of the mean-field. We then showed that the limiting diffusion process is an OU process that is controlled by the mean-field limit. We then showed how  Little's law and the FCLT can be combined to show that $P_{block}^{(N)}-\pi_C^d$ is $O(N^{-\frac{1}{2}})$ if $\beta\neq 0$ with the constants given in terms of the mean field and the parameters.
	These techniques may be used to study other routing strategies such as a randomized join below threshold strategy, or more generally other occupancy based routing strategies for which the global asymptotic stability of the mean-field can be established (see \cite{thiru_itc30} for such models).

	
	\section*{Acknowledgement}
	This research was supported in part by a Discovery Grant from the Natural Sciences and Engineering Research Council of Canada (NSERC). 
	Cette recherche a \'{e}t\'{e} subventionne\'{e} par le Conseil de recherches en sciences naturelles et g\'{e}nie du Canada (CRSNG).

\bibliographystyle{apt}
\bibliography{reference_loss_heavy_traffic}

\begin{thebibliography}{10}

\bibitem{AmazonEC2}
Amazon {EC2}.
\newblock \url{http://aws.amazon.com/ec2/}.

\bibitem{Azure}
Microsoft {A}zure.
\newblock \url{http://www.microsoft.com/windowsazure/}.

\bibitem{asmussen}
{\sc Asmussen, S.} (2003).
\newblock {\em Applied Probability and Queues} vol.~51 of {\em Stochastic
  Modelling and Applied Probability}.
\newblock Springer, New York.

\bibitem{gamarnik}
{\sc Eschenfeldt, P. and Gamarnik, D.} (2018).
\newblock Join the shortest queue with many servers. {T}he heavy-traffic
  asymptotics.
\newblock {\em Math. Oper. Res.\/} {\bf 43,} 867--886.

\bibitem{Ethier_Kurtz_book}
{\sc Ethier, S.~N. and Kurtz, T.~G.} (1985).
\newblock {\em Markov Processes: Characterization and Convergence}.
\newblock John Wiley and Sons Ltd.

\bibitem{gast_fns}
{\sc Gast, N.} (2017).
\newblock Expected values estimated via mean-field approximation are
  1/n-accurate.
\newblock {\em Proceedings of the ACM on Measurement and Analysis of Computing
  Systems\/} {\bf 1,} 1--26.

\bibitem{gast}
{\sc Gast, N. and Van~Houdt, B.} (2017).
\newblock A refined mean field approximation.
\newblock {\em Proc. ACM Meas. Anal. Comput. Syst.\/} {\bf 1,} 33:1--33:28.

\bibitem{gazd}
{\sc Gazdzicki, P., Lambadaris, I. and Mazumdar, R.} (1993).
\newblock Blocking probabilities for large multi-rate{ E}rlang loss systems.
\newblock {\em Adv.Appl.Prob.\/} {\bf 25,} 997--1009.

\bibitem{Graham_chaos}
{\sc Graham, C.} (2000).
\newblock Chaoticity on path space for a queueing network with selection of the
  shortest queue among several.
\newblock {\em J. Appl. Probab.\/} {\bf 37,} 198--211.

\bibitem{Graham_clt}
{\sc Graham, C.} (2005).
\newblock Functional central limit theorems for a large network in which
  customers join the shortest of several queues.
\newblock {\em Probability Theory and Related Fields\/} {\bf 131,} 97--120.

\bibitem{Karthik}
{\sc Karthik, A., Mukhopadhyay, A. and Mazumdar, R.~R.} (2017).
\newblock Choosing among heterogeneous server clouds.
\newblock {\em Queueing Syst.\/} {\bf 85,} 1--29.

\bibitem{Ledermann}
{\sc Ledermann, W., Reuter, G. E.~H. and Mahler, K.} (1954).
\newblock Spectral theory for the differential equations of simple birth and
  death processes.
\newblock {\em Philosophical Transactions of the Royal Society of London.
  Series A, Mathematical and Physical Sciences\/} {\bf 246,} 321--369.

\bibitem{mit}
{\sc Mitzenmacher, M.} (1996).
\newblock The power of two choices in randomized load balancing.
\newblock {\em PhD Thesis, Berkeley\/}.

\bibitem{Mukherjee_blocking}
{\sc Mukherjee, D., Borst, S.~C., van Leeuwaarden, J. and Whiting, P.~A.}
  (2016).
\newblock Asymptotic optimality of power-of-\$d\$ load balancing in large-scale
  systems.
\newblock In {\em arXiv:1612.00723}.

\bibitem{arpan}
{\sc Mukhopadhyay, A., Mazumdar, R.~R. and Guillemin, F.} (2015).
\newblock The power of randomized routing in heterogeneous loss systems.
\newblock In {\em Teletraffic Congress (ITC 27), 2015 27th International}.
\newblock pp.~125--133.

\bibitem{arpan2}
{\sc Mukhopadhyayay, A., Karthik, A., Mazumdar, R.~R. and Guillemin, F.~M.}
  (2015).
\newblock Mean field and propagation of chaos in multi-class heterogeneous loss
  models.
\newblock {\em Performance Evaluation\/} {\bf 91,} 117--131.

\bibitem{pang}
{\sc Pang, G., Talreja, R. and Whitt, W.} (2007).
\newblock Martingale proofs of many-server heavy-traffic limits for markovian
  queues.
\newblock {\em Probab. Surveys\/} {\bf 4,} 193--267.

\bibitem{thiru_itc30}
{\sc Vasantam, T. and Mazumdar, R.~R.}
\newblock On occupancy based randomized routing schemes in large systems of
  shared servers.
\newblock In {\em Proceedings 30th International Teletraffic Congress (ITC)}.

\bibitem{thiru_itc31}
{\sc Vasantam, T. and Mazumdar, R.~R.} (2019).
\newblock Fluctuations around the mean-field for a large scale {E}rlang loss
  system under the {SQ(d)} load balancing.
\newblock In {\em Proceedings of the 31st International Teletraffic Congress
  (ITC 31)}.

\bibitem{Vvedenskaya_inftran}
{\sc Vvedenskaya, N.~D., Dobrushin, R.~L. and Karpelevich, F.~I.} (1996).
\newblock Queueing system with selection of the shortest of two queues: an
  asymptotic approach.
\newblock {\em Problems of Information Transmission\/} {\bf 32,} 20--34.

\bibitem{Whitt_Loss_heavytraffic}
{\sc {Whitt}, W.} (1984).
\newblock Heavy-traffic approximations for service systems with blocking.
\newblock {\em AT\& T Bell Laboratories Technical Journal\/} {\bf 63,}
  689--708.

\bibitem{xie}
{\sc Xie, Q., Dong, X., Lu, Y. and Srikant, R.} (2015).
\newblock Power of d choices for large-scale bin packing: A loss model.
\newblock In {\em Proceedings of the 2015 ACM SIGMETRICS}.
\newblock pp.~321--334.

\bibitem{Ying}
{\sc Ying, L.} (2016).
\newblock On the approximation error of mean-field models.
\newblock In {\em Proceedings of the 2016 ACM SIGMETRICS International
  Conference on Measurement and Modeling of Computer Science}.
\newblock SIGMETRICS '16.
\newblock ACM, New York, NY, USA.
\newblock pp.~285--297.

\end{thebibliography}
\end{document}